\documentclass[11pt]{amsart}

\usepackage{amsmath, amscd, amssymb}
\usepackage[frame,cmtip,arrow,matrix,line,graph,curve]{xy}
\usepackage{graphpap, color}
\usepackage[mathscr]{eucal}
\usepackage{cancel}
\usepackage{verbatim}
 
\numberwithin{equation}{section}

\def\sO{{\mathscr O}}

\newcommand{\CC}{\mathbb{C}}

\newcommand{\PP}{\mathbb{P}}

\newcommand{\ZZ}{\mathbb{Z}}

\newcommand{\cal}{\mathcal}

\def\cE{{\cal E}}

\def\cM{{\cal M}}

\def\cW{{\cal W}}
\def\cW{{\cal W}}

\def\fC{\mathfrak{C}}

\def\fE{\mathfrak{E}}

\def\fN{\mathfrak{N}}

\def\loc{\mathrm{loc}}

\let\fE=\cE

\def\and{\quad{\rm and}\quad}
\def\lra{\longrightarrow }
\def\mapright#1{\,\smash{\mathop{\lra}\limits^{#1}}\,}

\let\sub=\subset

\DeclareMathOperator{\Tor}{Tor}

\newtheorem{prop}{Proposition}[section]
\newtheorem{theo}[prop]{Theorem}
\newtheorem{lemm}[prop]{Lemma}
\newtheorem{coro}[prop]{Corollary}
\newtheorem{rema}[prop]{Remark}
\newtheorem{exam}[prop]{Example}

\newtheorem{defi}[prop]{Definition}

\newtheorem{assu}[prop]{Assumption}

\def\beq{\begin{equation}}
\def\eeq{\end{equation}}

\def\virt{^{\mathrm{vir}} }

\def\bbL{\mathbb{L} }
\def\DM{Deligne-Mumford }

\def\Po{\PP^1}

\def\bl{\bigl(}
\def\br{\bigr)}

\def\bfc{\mathfrak{C} }

\def\ch{\mathrm{ch} }
\def\td{\mathrm{td} }
\def\coker{\mathrm{coker} }

\def\tE{\widetilde{E} }
\def\tC{\widetilde{C} }
\def\tX{\widetilde{X} }
\def\tF{\widetilde{F} }
\def\tsi{\tilde{\sigma} }

\def\trho{\tilde{\rho} }

\title{Localizing virtual structure sheaves by cosections}

\author{Young-Hoon Kiem}
\address{Department of Mathematics and Research Institute of Mathematics, Seoul National University, Seoul 08826, Korea} \email{kiem@snu.ac.kr}

\author{Jun Li}
\address{Department of Mathematics,  Stanford University, CA 94305, USA}
\email{jli@math.stanford.edu}

\thanks{YHK was partially supported by Samsung Science and Technology Foundation grant SSTF-BA1601-01; JL was partially supported by National Science Foundation, DMS-1601211}

\date{2018.9.1}

\begin{document}
\begin{abstract} 
We construct a cosection localized virtual structure sheaf when a \DM stack is equipped with a perfect obstruction theory and a cosection of the obstruction sheaf. 
\end{abstract}
\maketitle  

\section{Introduction}\label{sec1}

According to Schubert \cite{Klei}, enumerative geometry is about finding the number of geometric figures of fixed type satisfying certain given conditions. A typical way of solving an enumerative problem consists of constructing a moduli space parameterizing all geometric figures of fixed type and then finding the number of intersection points of the subsets defined by the given conditions. The latter part is called an intersection theory and a solid mathematical theory was developed by Fulton and MacPherson \cite{Fulton} in 1970s through a systematic use of normal cones and Gysin maps. 
In particular, the intersection rings for smooth schemes were rigorously defined and Riemann-Roch theorems were established for schemes. Fulton's intersection theory was updated for \DM stacks by Vistoli \cite{Vistoli} and for Artin stacks by Kresch \cite{Kres}. 

However moduli spaces are often very singular and do not behave well under deformation. To deal with this issue, the theory of virtual fundamental classes was developed in 1990s by Li-Tian \cite{LiTi} and Behrend-Fantechi \cite{BeFa}. A \DM stack $X$ has its intrinsic normal cone $\bfc_X$ which is locally defined as the quotient stack $C_{U/M}/{T_M|_U}$ for an \'etale $U\to X$ and a closed embedding $U\hookrightarrow M$ into a smooth $M$. When $\bfc_X$ is embedded into a vector bundle stack $\fE_X$ which is locally the quotient $E_1/E_0$ for vector bundles $E_1, E_0$, the virtual fundamental class is defined as the intersection 
$$[X]\virt=0^!_{\fE_X}[\bfc_X]\in A_*(X)$$
of $\bfc_X$ with the zero section of $\fE_X$. Many nice properties such as deformation invariance can be deduced under reasonable assumptions \cite{KKP} and since 1995, important enumerative invariants in algebraic geometry have been constructed as integrals on the virtual cycles $[X]\virt$ on suitable moduli spaces $X$, including the Gromov-Witten, Donaldson-Thomas, Pandharipande-Thomas invariants and more. 

The computation of these virtual invariants is known to be very hard because there are not many tools to handle the virtual cycle $[X]\virt$. When there is an action of $\CC^*$ on $X$, under suitable assumptions, the virtual cycle is localized to the fixed locus $X^{\CC^*}$ by the torus localization formula (cf. \cite{GrPa})
$$[X]\virt=\frac{[X^{\CC^*}]\virt}{e(N\virt)}$$
which has been most effective for actual computations so far. Recently, another localization of the virtual cycle $[X]\virt$ was discovered \cite{KLc} which also turned out to be quite useful \cite{CK, CLp, CLL, CLLL, Clad, FJR, GS1, GS2, HLQ, JiTh, KL1, KL2, KoTh, MPT, PT}. It says that when there is a morphism $\sigma:\fE_X\to \sO_X$, called a cosection, the virtual cycle is localized to the zero locus $X(\sigma)$ of $\sigma$, i.e. we have a class $[X]\virt_\loc\in A_*(X(\sigma))$ which equals $[X]\virt$ when pushed forward to $X$. The two localizations were combined in \cite{CKL}.

Recently there arose a demand to lift the theory of virtual cycles in Chow groups to algebraic K-groups with applications towards physics and geometric representation theory. In \cite{BeFa, YPLe}, the K-theoretic virtual fundamental class was defined as 
$$[\sO_X\virt]=0^!_{E_1}[\sO_{C_1}]=[\sO_X\otimes^L_{\sO_{E_1}}\sO_{C_1}] \in K_0(X)$$
and called the \emph{virtual structure sheaf} of $X$.
Here % $K_0(X)$ is the Grothendieck group of coherent sheaves on $X$ with relations $[F]=[F']+[F'']$ whenever we have an exact sequence $0\to F'\to F\to F''\to 0$;
$C_1=\bfc_X\times_{\fE_X}E_1$ when $\fE_X=E_1/E_0$ is a global resolution by vector bundles. 
In order to lift some of the results on the virtual cycles and invariants to the setting of algebraic K-groups, it is necessary to develop K-theoretic techniques to handle the virtual structure sheaf $[\sO_X\virt]$ such as the cosection localization. In this paper, we prove the following.

\begin{theo} (Theorem \ref{n8}, Proposition \ref{3.65}, Proposition \ref{3.66})
Let $X$ be a \DM stack equipped with a perfect obstruction theory $\phi:E\to \bbL_X$ and a cosection $\sigma:h^1(E^\vee)\to \sO_X$ whose zero locus is denoted by $X(\sigma)$. Assume $E$ admits a global resolution $[E^{-1}\to E^0]$ by locally free sheaves (e.g. if $X$ is quasi-projective). Then there is a cosection localized  virtual structure sheaf 
$$[\sO\virt_{X,\loc}]\in K_0(X(\sigma))$$
such that $\imath_*[\sO\virt_{X,\loc}]=[\sO\virt_X]\in K_0(X)$ where $\imath:X(\sigma)\to X$ denotes the inclusion. Moreover, 
$[\sO\virt_{X,\loc}]$ is independent of the choice of the resolution $[E^{-1}\to E^0]$ and is
deformation invariant. 
\end{theo}

By \cite[\S4]{KLc}, the cone $C_1\subset E_1$ has (reduced) support in 
$$E_1(\sigma)=E_1|_{X(\sigma)}\cup \ker(\sigma:E_1|_U\to  \sO_U)$$
where $U=X-X(\sigma)$. We define the cosection localized Gysin map (cf. Theorem \ref{n10})
$$0^!_{E_1,\sigma}:K_0(E_1(\sigma))\to K_0(X(\sigma))$$ and 
the cosection localized virtual structure sheaf $[\sO\virt_{X,\loc}]$ is defined as
$$[\sO\virt_{X,\loc}]=0^!_{E_1,\sigma}[\sO_{C_1}]\in K_0(X(\sigma)).$$

Using the cosection localized virtual structure sheaf $[\sO\virt_{X,\loc}]\in K_0(X(\sigma))$, we can define the cosection localized virtual Euler characteristic, even when $X$ is not proper, as long as $X(\sigma)$ is proper. 
\begin{defi}
The cosection localized \emph{virtual Euler characteristic} of a class $\beta\in K^0(X)$ is defined as
\[\chi\virt_\loc (X,\beta)=\chi(X(\sigma),\beta\cdot\sO\virt_{X,\loc})=\sum_i (-1)^i \dim H^i(X(\sigma),\beta\cdot \sO\virt_{X,\loc}).\]
\end{defi}

As an application, we lift the results of \cite{CLp, CLL} to the K-theoretic setting. In particular, we define the K-theoretic Fan-Jarvis-Ruan-Witten invariant for narrow sector (which doesn't involve the matrix factorization machinery of Polishchuk and Vaintrob)  as the Euler characteristic of the virtual structure sheaf on the moduli space of spin curves with sections. 

\medskip

The layout of this paper is as follows. In \S2, we collect useful facts. In \S3, we prove that the virtual structure sheaf vanishes if there is a surjective cosection $\sigma: Ob_X\to \sO_X$. In \S4, we construct the cosection localized Gysin map $0^!_{E_1,\sigma}$. In \S5, we define the cosection localized virtual structure sheaf $[\sO\virt_{X,\loc}]$ and prove key properties. In \S6, we apply the cosection localized virtual structure sheaf to construct a K-theoretic FJRW invariant and discuss the K-theoretic Landau-Ginzburg/Calabi-Yau correspondence. 

\medskip

In this paper, all schemes or \DM stacks are separated and Noetherian of finite type over the complex number field $\CC$. When $X\hookrightarrow Y$ is a closed embedding, its normal cone is denoted by $C_{X/Y}$. When $f:X\to Y$ is a morphism and $F$ is a coherent sheaf on $Y$, $F|_Y$ denotes the (underived) pullback $f^*F$. 

\medskip

We thank Richard Thomas for useful comments including Remark \ref{3.rt}.
\bigskip

\def\lred{_{\mathrm{red}} }
\def\Tor{\mathrm{Tor} }
\section{Preliminaries}\label{sec2}
In this section, we collect useful facts about algebraic K-groups and cosection localization.

\subsection{Algebraic K-theory}\label{S2.1}

For a \DM stack $X$, $K_0(X)$ (resp. $K^0(X)$) denotes the Grothendieck group generated by coherent sheaves (resp. locally free sheaves) on $X$ with relations $[F]=[F']+[F'']$ whenever we have an exact sequence $0\to F'\to F\to F''\to 0.$

For a projective morphism $f:X\to Y$, the pushforward
$$f_*:K_0(X)\lra K_0(Y)$$
is defined by the right derived direct image $Rf_*$ as
$$f_*[F]=[Rf_*F]=\sum_i(-1)^i[R^if_*F].$$
The pullback $$f^*:K_0(Y)\lra K_0(X)$$
for a morphism $f:X\to Y$ is defined by the left derived inverse image $Lf^*$ as
$$f^*[G]=[Lf^*G]=[\sO_X\otimes^L_{f^{-1}\sO_Y}f^{-1}G]=\sum_i(-1)^i[\Tor_i^{f^{-1}\sO_Y}(\sO_X,f^{-1}G)]$$
when the sum is finite, i.e. $\Tor_i^{f^{-1}\sO_Y}(\sO_X,f^{-1}G)\ne 0$ for only finitely many $i$.  
Tensor product of locally free sheaves makes $K^0(X)$ a commutive ring and $K_0(X)$ a module over $K^0(X)$. 

For a \DM stack $X$, let $X\lred$ be the reduced stack of $X$ and $\imath:X\lred\to X$ denote the inclusion. 
\begin{prop}\label{n1}
For a \DM stack $X$, $\imath_*:K_0(X\lred)\to K_0(X)$ is an isomorphism. \end{prop}
\begin{proof}
Let $I$ be the ideal sheaf defining $X\lred$ in $X$, i.e. it is the ideal of nilpotents. Since $X$ is Noetherian by assumption, $I^n=0$ for some $n>0$. Hence, if $F$ is a coherent sheaf on $X$, 
\[ [F]=\sum_{r=0}^{n-1}[I^rF/I^{r+1}F]\in K_0(X) \]
and $I^rF/I^{r+1}F$ are coherent sheaves on $X\lred$. Hence $\imath_*$ is surjective. In fact, the assignment
\beq\label{n2} [F]\mapsto \sum_{r=0}^{n-1}[I^rF/I^{r+1}F]\eeq
is obviously a right inverse of $\imath_*$ and hence $\imath_*$ is injective. 
It is a straightforward exercise to check that \eqref{n2} is a homomorphism of $K$-groups.
%Therefore the proof will be complete once we show that \eqref{n2} defines a homomorphism of $K_0(X)$ into $K_0(X\lred)$. This follows from the following lemma.
\end{proof}

%\begin{lemm}\label{n3}
%Under the notation of Proposition \ref{n1}, if $0\to F'\to F\to F''\to 0$ is an exact sequence of coherent sheaves on $X$, we have an equality
%$$\sum_r[I^rF/I^{r+1}F]=\sum_r[I^rF'/I^{r+1}F']+\sum_r[I^rF''/I^{r+1}F''] \in K_0(X\lred).$$
%\end{lemm}
%\begin{proof}
%Let $F_r=I^rF$ and $F_r''=I^rF''$. Let
%$$F'_r=\ker(I^rF\to I^rF'')$$
%so that we have a commutative diagram of exact sequences
%\[\xymatrix{
%0\ar[r] & F'_{r+1}\ar[r]\ar@{^(->}[d] & F_{r+1}\ar@{^(->}[d]\ar[r] & F''_{r+1}\ar[r]\ar@{^(->}[d] & 0 \\
%0\ar[r] & F'_{r}\ar[r] & F_{r}\ar[r] & F''_{r}\ar[r] & 0 
%}\]
%which gives us an equality
%$$\sum_r[I^rF/I^{r+1}F]-\sum_r[I^rF''/I^{r+1}F''] = \sum_r [F_r'/F'_{r+1}].$$
%Hence it suffices to prove that 
%\beq\label{n4} \sum_r [F_r'/F'_{r+1}]=\sum_r[I^rF'/I^{r+1}F']. \eeq
%
%Since $F'=F'_0\supset F'_1\supset \cdots\supset F'_n=0$ and $F'=I^0F'\supset IF'\supset \cdots \supset I^nF'=0$ are two filtrations of $F'$, they have a common refinement by Schreier's theorem (cf. \cite{Lang}) and hence we have the equality \eqref{n4}. 
%\end{proof} 

\begin{prop}\label{n5}\cite[Lemma 17]{BoSe} For a \DM stack $X$, $K_0(X)$ is generated by the classes $[\sO_Z]$ for integral closed substacks $Z$ in $X$.
\end{prop}
\begin{proof}
By Proposition \ref{n1}, we may assume $X$ is reduced and consider only sheaves supported on reduced stacks. Now the proof of \cite[Lemma 17]{BoSe}, by induction on the dimension of the support, also proves the proposition.
\end{proof}

\begin{prop}\label{n6} \cite[Proposition 7]{BoSe}
Let $X$ be a \DM stack and $Z$ be a closed substack. Let $U=X-Z$. Then the inclusions
\[ \xymatrix{Z\ar@{^(->}[r]^\imath & X & U\ar@{_(->}[l]_{\jmath} }\]
induce a complex
\beq\label{n7} K_0(Z)\mapright{\imath_*} K_0(X) \mapright{\jmath^*} K_0(U)\lra 0\eeq
which is exact at $K_0(U)$. %\red{Suppose $X$ has quasi-projective coarse moduli, 
If every coherent sheaf on $U$ extends to a coherent sheaf on $X$ (e.g. $X$ is a scheme), 
then \eqref{n7} is exact. 
\end{prop}
\begin{proof}
$\jmath^*$ is surjective by Proposition \ref{n5} since $\jmath^*[\sO_{\bar Z}]=[\sO_Z]$ for an integral $Z\subset U$ where $\bar Z$ denotes the closure of $Z$ in $X$. When every coherent sheaf on $U$ extends to a coherent sheaf on $X$, the exactness of \eqref{n7} is proved by the same proof of \cite[Proposition 7]{BoSe}.
\end{proof}

For instance, when $X$ has the resolution property, every sheaf of $\sO_U$-modules, where $U\sub X$ open,
can be extended to a sheaf of $\sO_X$-modules. When the Keel-Mori coarse moduli of $X$ is a separated
scheme, $X$ has the resolution property \cite{Totaro}.

\subsection{Virtual structure sheaf}\label{S2.2}

Let $X$ be a \DM stack equipped with a perfect obstruction theory $\phi:E\to \bbL_X$, i.e. $\phi$ is a morphism in the derived category of quasi-coherent sheaves on $X$ such that $h^0(\phi):h^0(E)\to h^0(\bbL_X)$ is an isomorphism, $h^{-1}(\phi):h^{-1}(E)\to h^{-1}(\bbL_X)$ is surjective, and $E$ is locally isomorphic to a 2-term complex of locally free sheaves, concentrated in degrees $-1$ and $0$. Here $\bbL_X=\tau^{\ge -1}L_X$ is the truncated cotangent complex of $X$. 

To avoid discussion on algebraic K-theory of Artin stacks, let us suppose in this paper that  
$E$ admits a global resolution by a 2-term complex $[E^{-1}\to E^0]$ of locally free sheaves on $X$. Let $[E_0\to E_1]$ denote the dual with $E_i=(E^{-i})^\vee$. By \cite{BeFa}, the intrinsic normal cone $\bfc_X$ of $X$ 
is locally the Artin stack $C_{X/M}/{T_M|_X}$ if $X\to M$ is a closed embedding into a smooth \DM stack. Here $C_{X/M}$ denotes the normal cone of $X$ in $M$. 
Then $h^1/h^0(\phi^\vee)$ embeds $\bfc_X$ into $\fE_X=E_1/E_0$. The intrinsic normal cone $\bfc_X$
lifts to a cone $$C_1:=\bfc_X\times_{\fE_X}E_1\subset E_1$$ and the virtual fundamental class is defined as 
$$[X]\virt=0^!_{\fE_X}[\bfc_X]=0^!_{E_1}[C_1].$$
Likewise, the virtual structure sheaf of $X$ is defined as the sheaf theoretic intersection
\beq\label{2.45} [\sO_X\virt]=[\sO_X\otimes^L_{\sO_{E_1}}\sO_{C_1}]=\sum_i(-1)^i [\mathrm{Tor}_i^{E_1}(\sO_X,\sO_{C_1})] \in K_0(X)\eeq
where $\otimes^L_{\sO_{E_1}}$ is the derived tensor product (cf. \cite[Remark 5.4]{BeFa}). 
If we denote the bundle projection $E_1\to X$ by $\pi$, the tautological section of $\pi^*E_1$ over $E_1$ gives us the Koszul resolution $\wedge^\cdot \pi^*E_1^\vee$ of the sheaf $\sO_X$ on $E_1$. Hence \eqref{2.45} can be rephrased as
\beq\label{2.46} [\sO_X\virt]=[\wedge^\cdot \pi^*E_1^\vee\otimes_{\sO_{E_1}} \sO_{C_1}].\eeq

\subsection{Localizing virtual fundamental classes by cosections}\label{S2.3}

Suppose the obstruction sheaf $Ob_X:=h^1(E^\vee)$ admits a cosection $\sigma:Ob_X\to\sO_X$. 
Let $X(\sigma)$ denote the locus where $\sigma$ is not surjective, i.e. the closed stubstack defined by the ideal sheaf $\sigma(Ob_X)\subset \sO_X$.  
By \cite[\S4]{KLc}, the (reduced) support of $C_1$ is contained in 
$$E_1(\sigma)=E_1|_{X(\sigma)}\cup \mathrm{ker}(\sigma:E_1\to Ob_X\to \sO_X).$$

The ordinary Gysin map $0^!_{E_1}:A_*(E_1)\to A_*(X)$ can be localized to a homomorphism (cf. \cite[\S2]{KLc})
$$0^!_{E_1,\sigma}:A_*(E_1(\sigma))\lra A_*(X(\sigma))$$
such that $\imath_*\circ 0^!_{E_1,\sigma}=0^!_{E_1}$ where $\imath:X(\sigma)\to X$ is the inclusion map. 

The cosection localized virtual cycle is then obtained by 
$$[X]\virt_\loc=0^!_{E_1,\sigma}[C_1]\in A_*(X(\sigma))$$
which satisfies many nice properties such as $\imath_*[X]\virt_\loc=[X]\virt$ and deformation invariance. 
When $X$ is not proper, $[X]\virt$ is not properly supported in general. If $X(\sigma)$ is proper, then we can still define integrals on the cosection localized virtual fundamental class $[X]\virt_\loc$. 

In the subsequent sections, we will show that the virtual structure sheaf $[\sO_X\virt]\in K_0(X)$ is also localized to $X(\sigma)$ by the cosection.

\section{A vanishing result for virtual structure sheaves}\label{S3.1}

In this section, we show that if there is a surjective cosection 
$\sigma:Ob_U\to \sO_U$ on a \DM stack $U$, then the virtual structure sheaf $[\sO_U\virt]\in K_0(U)$ is zero. 

Let $U$ be a \DM stack and   
let $\phi:E_U\to \bbL_U$ be a perfect obstruction theory on $U$ where $\bbL_U=\tau^{\ge -1}L_U$ is the truncated cotangent complex on $U$, i.e. $h^0(\phi)$ is an isomorphism and $h^{-1}(\phi)$ is surjective. By \cite{BeFa}, we have a closed embedding
\[ \fN_U=h^1/h^0(\bbL_U^\vee)\hookrightarrow h^1/h^0(E_U^\vee)=:\fE_U\]
of the virtual normal sheaf $\fN_U$ into the vector bundle stack $\fE_U$. In this section, we make following assumptions.

\begin{assu}\label{3.1}
(1) There is a surjective homomorphism $\sigma_U:Ob_U\to \sO_U$ where $Ob_U=h^1(E^\vee_U)$ is the obstruction sheaf.\\
(2) There is a 2-term complex $[E^{-1}\to E^0]$ of locally free sheaves on $U$ which is isomorphic to $E_U$ in the derived category $D(\sO_U)$. 
\end{assu}

Let $E_i=(E^{-i})^\vee$ be the dual bundle of $E^{-i}$ for $i=0,1$ so that $E^\vee_U\cong [E_0\to E_1]$. Then $\fE_U=h^1/h^0(E^\vee_U)$ is the quotient stack $[E_1/E_0]$. 
Since the abelian hull of the intrinsic normal cone $\bfc_U$ is the intrinsic normal sheaf $\fN_U$, we have an embedding
$$\bfc_U\hookrightarrow \fN_U\hookrightarrow \fE_U$$
uniquely determined by the perfect obstruction theory $\phi$.
The natural homomorphism $E_1\to \fE_U=[E_1/E_0]$ induces the Cartesian square
\[\xymatrix{
C_1\ar@{^(->}[r] \ar[d] &E_1\ar[d]\\
\bfc_U\ar@{^(->}[r] &\fE_U .
}\]
Since $Ob_U$ is the cokernel of $E^\vee=[E_0\to E_1]$, we have a surjective homomorphism
\beq\label{3.5} E_1\lra Ob_U\mapright{\sigma_U} \sO_U .\eeq
By \cite[\S4]{KLc}, $C_1$ has support in the subbundle 
$$E'_1=\mathrm{ker}(E_1\twoheadrightarrow \sO_U).$$
By Proposition \ref{n1},  we may think of $\sO_{C_1}$ as a sheaf on $E_1'$ because 
\beq\label{3.60} [\sO_{C_1}]\in K_0(E_1').\eeq

For the computation of $\sO\virt_U$, we may use the Koszul resolution of $\sO_U$ as an $\sO_{E_1}$-module. If we denote the vector bundle projection of $E_1$ by $\pi:E_1\to U$, the tautological section $\sO_{E_1}\to \pi^*E_1$ induces a complex
\[
0\lra \wedge^r\pi^*E_1^\vee \lra \cdots \lra \wedge^2\pi^*E_1^\vee \lra \pi^*E_1^\vee \lra \sO_{E_1}\lra 0
\]
where $r=\mathrm{rank}(E_1)$. As is well known, this complex is a locally free resolution of $\sO_U$ on $E_1$, and $\mathrm{Tor}^{E_1}_i(\sO_{C_1}, \sO_U)$ is the $(-i)$-th cohomology of the complex
\[
0\to \wedge^r\pi^*E_1^\vee\otimes_{\sO_{E_1}}\sO_{C_1} \to \cdots \to \wedge^1\pi^*E_1^\vee\otimes_{\sO_{E_1}}\sO_{C_1} %\to \pi^*E_1^\vee\otimes_{\sO_{E_1}}\sO_{C_1} 
\to \sO_{C_1}\to 0.
\]
In this section, we prove the following.
\begin{prop}\label{3.2}
Let $\sigma_U:Ob_U\to \sO_U$ be a surjective homomorphism. Then the virtual structure sheaf $[\sO_U\virt]\in K_0(U)$ is zero. 
\end{prop}
\begin{proof}
Let $E_1'$ be the kernel of the surjective homomorphism \eqref{3.5} so that we have an exact sequence of vector bundles
\[ 0\lra E_1'\lra E_1\lra \sO_U\lra 0 \]
which induces an exact sequence of vector bundles
\beq\label{3.6} 0\lra \wedge^{i-1}{E'_1}^\vee \lra \wedge^{i}E_1^\vee \lra \wedge^{i}{E'_1}^\vee \lra 0 \eeq
for each $i$, 
where the superscript $\vee$ denotes the dual vector bundle. 
Recall that the normal cone $C_1$ has support in $E_1'$ by \cite{KLc}. 
Pulling back \eqref{3.6} by $\pi$ and tensoring with $\sO_{C_1}$ give us a short exact sequence
\beq\label{3.3}
 0\lra \wedge^{\cdot-1}\pi^*{E'_1}^\vee\otimes_{\sO_{E_1}}\sO_{C_1} \lra \wedge^{\cdot}\pi^*E_1^\vee\otimes_{\sO_{E_1}}\sO_{C_1} \lra \wedge^{\cdot}\pi^*{E'_1}^\vee\otimes_{\sO_{E_1}}\sO_{C_1} \lra 0
\eeq
of complexes of coherent sheaves on $C_1\subset E_1'$. Then we have 
$$\wedge^\cdot\pi^*{E_1'}^\vee\otimes_{\sO_{E_1}}\sO_{C_1}\cong (\wedge^\cdot \pi^*{E_1'}^\vee\otimes_{\sO_{E_1}}\sO_{E_1'})\otimes_{\sO_{E_1'}}\sO_{C_1}=\wedge^\cdot {\pi'}^*{E_1'}^\vee\otimes_{\sO_{E_1'}}\sO_{C_1}$$
where $\pi':E_1'\to U$ is the bundle projection. 
Let $$A_{j}=\mathrm{Tor}^{E_1}_j(\sO_U,\sO_{C_1}) \and A'_{j}=\mathrm{Tor}_j^{E'_1}(\sO_U,\sO_{C_1})$$
so that 
\beq\label{3.7} \sO\virt_U=\sum_j (-1)^j A_{j} .\eeq
The long exact sequence associated to \eqref{3.3} gives us the exact sequence  
\beq\label{3.4}
0\lra A_{r}\lra A'_{r}\lra A'_{r-1}\lra \cdots \lra A'_{0}\lra A_{1}\lra A'_{1}\lra 0
\eeq
and an isomorphism $A_{0}=A'_{0}$ which imply
$$\sum_j (-1)^jA_{j}=\sum_j (-1)^jA'_{j-1}+\sum_j (-1)^jA'_{j}=0.$$
Therefore we have the vanishing
\beq\label{3.8}
[\sO\virt_U]=\sum_j (-1)^jA_{j}=0\in K_0(U).
\eeq
This proves the proposition.
\end{proof}

\begin{rema}\label{3.rt}
Richard Thomas kindly informed us of a simple proof of Proposition \ref{3.2} as follows.
The normal bundle $N$ of $E_1'$ in $E_1$ is the trivial line bundle. Letting $\imath:E'\to E$ 
denote the inclusion, for a coherent sheaf $\alpha$ on $E_1'$  like $\sO_{C_1}$, the exact triangle
$$N^\vee\otimes\alpha[-1]\to L\imath^*\imath_*\alpha\to \alpha \to N^\vee\otimes\alpha$$
shows that $[L\imath^*\imath_*\alpha]=[\sO_{E_1'}\otimes^L_{\sO_{E_1}}\alpha]=0\in K_0(E_1')$. 
This certainly suffices for the vanishing $[\sO_U\virt]=[\sO_U\otimes^L_{\sO_{E_1}}\sO_{C_1}]=0.$
\end{rema}

In many natural situations, the cosection $\sigma$ is not surjective. 
Let $X(\sigma)$ be the zero locus of a cosection $\sigma:Ob_X\to \sO_X$ of the obstruction sheaf $Ob_X$ on a \DM stack $X$.
Then $\sigma$ is surjective on $U=X-X(\sigma)$ and hence we have the vanishing
$$\jmath^*[\sO_X\virt]=[\sO_U\virt]=0\in K_0(U)$$
where $\jmath:U\to X$ denotes the inclusion. By \eqref{n7}, when $X$ has the resolution property, we should have
$$\imath_*(\xi)=[\sO_X\virt],\quad \exists \xi\in K_0(X(\sigma))$$
where $\imath:X(\sigma)\to X$ is the inclusion. 
In the subsequent sections, we will show that the virtual structure sheaf $[\sO_X\virt]$ localizes to the substack $X(\sigma)$ in a canonical manner. 

%\begin{rema} Let $\imath:X(\sigma)\to X$ denote the inclusion map of the zero locus of the cosection $\sigma:Ob_X\to \sO_X$. 
%When $\imath_*:K_0(X(\sigma))\to K_0(X)$ is injective, by Theorem \ref{2.10}, there is a unique
%$$[\sO\virt_{X,\loc}]\in K_0(X(\sigma))$$
%such that $\imath_*\sO\virt_{X,\loc}=\sO\virt_X$. Since $\sO_X\virt$ satisfies the deformation invariance property, so is the localized virtual structure sheaf $\sO_{X,\loc}\virt$.    
%\end{rema}

\section{Cosection localized Gysin map} \label{SclG}

Let $X$ be a \DM stack and $E_1$ be a vector bundle on $X$. Let 
$$\sigma:E_1\lra \sO_X$$
be a homomorphism of $\sO_X$-modules and $X(\sigma)$ be the zero locus of $\sigma$, i.e. the closed substack defined by the ideal sheaf $\sigma(E_1)\subset \sO_X$.

Let $U=X-X(\sigma)$ and let 
%$E'_1$ be the closure in $E_1$ of 
%$$\ker (\sigma:E_1|_U\twoheadrightarrow \sO_U),
%$$
%where the latter is a subbundle of $E_1|_U$ of corank 1. 
\beq\label{n37} E_1(\sigma):=E_1|_{X(\sigma)}\cup \ker (\sigma:E_1|_U\twoheadrightarrow \sO_U).\eeq
Here we are now considering $E_1$ to be the total space of the vector bundle $E_1$ and likewise $E_1(\sigma)$ is a substack of $E_1$. 

The purpose of this section is to construct the following cosection localized Gysin map.
\begin{theo}\label{n10} Under the notation as above, we have a homomorphism 
\beq\label{n11} 
0^!_{E_1,\sigma}:K_0(E_1(\sigma))\lra K_0(X(\sigma)) \eeq
which satisfies 
\beq\label{n12} \imath_*\circ 0^!_{E_1,\sigma}=0^!_{E_1}\circ\jmath_*:K_0(E_1(\sigma))\lra K_0(X) \eeq
where $\imath:X(\sigma)\to X$ and $\jmath:E_1(\sigma)\to E_1$ denote the inclusion maps while 
\beq\label{n13} 0^!_{E_1}[F]=[\sO_X\otimes^L_{\sO_{E_1}}F]=\sum_i(-1)^i\Tor_i^{E_1}(\sO_X,F).\eeq
\end{theo}

\def\tX{\tilde{X} }
\def\him{\hat{\imath} }
\begin{proof}
To motivate, let us assume that $X(\sigma)$ is a Cartier divisor $D$ so that we have an exact sequence
$$ 0\lra E_1'\lra E_1\mapright{\sigma} \sO_X(-D)\lra 0$$
of locally free sheaves. For any coherent sheaf $F$ on $E_1(\sigma)$, since $E_1'$ is a Cartier divisor of $E_1(\sigma)$, we have an exact sequence
$$0\lra G\lra F\lra \tF=F|_{E_1'}\lra 0$$
of coherent sheaves on $E_1(\sigma)$ with $[G]\in K_0(E|_{X(\sigma)})$ so that $[F]=[\tF]+[G]$. We define
$$0^!_{E_1,\sigma}[F]=D^\vee\cdot 0^!_{E_1'}[\tF]+0^!_{E|_D}[G]\in K_0(X(\sigma))$$
where $D^\vee\cdot :K_0(X)\to K_0(D)$ is defined by 
\beq\label{z1} D^\vee\cdot [F'] = [\sO_D^\vee\otimes^L_{\sO_{X}}F']=\sum_i (-1)^i [\Tor^{X}_i(\sO_D^\vee,F')]\eeq
%for the Cartier divisor $D$ of $X$. 
with $\sO_D^\vee=[\sO_X(-D)\to \sO_X]^\vee=[\sO_X\to \sO_X(D)]$. 

When $X(\sigma)$ is not a Cartier divisor, we 
let $\rho:\tX\to X$ be the blowup of $X$ along $X(\sigma)$ so that we have a surjective homomorphism
\beq\label{n14} \tsi:\tE_1\lra \sO_{\tX}(-D)\eeq
where $\tE_1=\rho^*E_1$ and $D$ is the exceptional divisor. 
Let $\tE_1'$ be the kernel of $\tsi$. Let $\rho':D\to X(\sigma)$ denote the restriction of $\rho$ to $D$ and $\trho:\tE_1'\to E_1(\sigma)$ be the map induced by the projection 
$$\tE_1=E_1\times_X\tX\lra E_1.$$
Let $\trho':\tE'_1|_D\to E_1|_{X(\sigma)}$ denote the restriction of $\trho$ to $\tE_1|_D$ and let 
$$\hat{\imath}:E_1|_{X(\sigma)}\to E_1(\sigma) \and \tilde{\imath}':\tE'_1|_D\to \tE'_1$$ 
denote the inclusion maps. 

For any coherent sheaf $F$ on $E_1(\sigma)$, let $\tF$ be any coherent sheaf on $\tE'_1$ such that there is an epimorphism
\beq\label{n15} \trho^*F\lra \tF\eeq
and that 
$$\tF|_{\tE'_1|_{\tX-D}}=F|_{\ker(\sigma|_U)}$$
where we identified $\tE'_1|_{\tX-D}$ with $\ker(\sigma|_U)$. 
For instance, we may choose $\tF=\trho^*F$.

By adjunction together with \eqref{n15}, we have homomorphisms
$$\eta_{F,\tF}:F\lra \trho_*\trho^*F\lra \trho_*\tF.$$
Since $F$, $\trho_*\trho^*F$ and $\trho_*\tF$ all coincide over $U$, we have
\[ [\ker\, \eta_{F,\tF}] - [\coker\, \eta_{F,\tF}]\in K_0(E_1|_{X(\sigma)}). \]
Since $\trho$ is an isomorphism over $U$, we have
\[ [R^i\trho_*\tF]\in K_0(E_1|_{X(\sigma)}) \]
for each $i\ge 1$. Let us denote the difference $F-R\trho_*\tF$ by  
\beq\label{n18} R_{F,\tF}:=[\ker\, \eta_{F,\tF}] - [\coker\, \eta_{F,\tF}] - \sum_{i\ge 1} (-1)^i [R^i\trho_*\tF]\in K_0(E_1|_{X(\sigma)}). \eeq

We now define the cosection localized Gysin map by 
\beq\label{n16} 0^!_{E_1,\sigma}[F]:=\rho'_*(D^\vee\cdot 0^!_{\tE'_1}\tF)+0^!_{E_1|_{X(\sigma)}}R_{F,\tF}\ \ \in K_0(X(\sigma))\eeq
where $D^\vee\cdot [F']=[\sO_D^\vee\otimes^L_{\sO_{\tX}}F']$ for $\sO_D^\vee=[\sO_{\tX}\to \sO_{\tX}(D)].$ 
To complete the proof, we will show the following:
\begin{enumerate}
\item[(i)] \eqref{n16} is independent of the choice of $\tF$.
\item[(ii)] If $0\to F'\to F\to F''\to 0$ is an exact sequence of coherent sheaves on $E_1(\sigma)$, then 
\beq\label{n17} 0^!_{E_1,\sigma}[F]=0^!_{E_1,\sigma}[F']+0^!_{E_1,\sigma}[F''].\eeq
\item[(iii)] \eqref{n12} holds.
\end{enumerate}

\def\tim{\tilde{\imath} }
To prove (i), since all $\tF$ are quotients of $\trho^*F$, it suffices to show
\beq\label{n19}
\rho'_*(D^\vee\cdot 0^!_{\tE'_1}\trho^*F)+0^!_{E_1|_{X(\sigma)}}R_{F,\trho^* F}=\rho'_*(D^\vee\cdot 0^!_{\tE'_1}\tF)+0^!_{E_1|_{X(\sigma)}}R_{F,\tF} .\eeq
%To simplify the notation, let us denote $\trho^*F|_{\tE_1'}$ by $\trho^*F$.
Let $G$ be the kernel of \eqref{n15} so that we have an exact sequence
$$0\lra \tim'_*G\lra \trho^*F\lra \tF\lra 0$$
with $[G]\in K_0(\tE'_1|_D)$. Then we have
$$\rho'_*(D^\vee\cdot 0^!_{\tE'_1}\trho^*F) - \rho'_*(D^\vee\cdot 0^!_{\tE'_1}\tF) = \rho'_*(D^\vee\cdot 0^!_{\tE'_1}\tim'_*G)$$
\beq\label{n20} = \rho'_*(D^\vee\cdot 0^!_{\tE'_1|_D}G)=\rho'_*( 0^!_{\tE_1|_D}G)=0^!_{E_1|_{X(\sigma)}}(\trho'_*[G]).\eeq
By \eqref{n20}, \eqref{n19} follows once we show 
\beq\label{n21} 0^!_{E_1|_{X(\sigma)}}R_{F,\trho^* F} =0^!_{E_1|_{X(\sigma)}}R_{F,\tF} -0^!_{E_1|_{X(\sigma)}}(\trho'_*[G]).\eeq
Obviously \eqref{n21} follows from 
\beq\label{n22} R_{F,\trho^* F} =R_{F,\tF} -\trho'_*[G]\in K_0(E_1|_{X(\sigma)}).\eeq

To prove \eqref{n22}, we consider the long exact sequence 
\beq\label{n24} 
0\to \trho'_*G\to \trho_*\trho^*F\to \trho_*\tF\to R^1\trho'_*G\to R^1\trho_*\trho^*F\to R^1\trho_*\tF\to \cdots\eeq
and break it into exact sequences as follows:
\beq\label{n25}0\lra \trho'_*G\lra \trho_*\trho^*F\lra  H\lra 0 %\trho_*\tF\lra L\lra 0
\eeq
\beq\label{n26}0\lra H\lra \trho_*\tF\lra L\lra 0\eeq
\beq\label{n27}0\lra L\lra R^1\trho'_*G\lra R^1\trho_*\trho^*F\lra R^1\trho_*\tF\lra \cdots\eeq
Note that \eqref{n27} consists only of sheaves on $E_1|_{X(\sigma)}$ and we have an equality
\beq\label{n28} \sum_{i\ge 1}(-1)^i[R^i\trho_*\trho^*F]=\sum_{i\ge 1}(-1)^i[R^i\trho_*\tF] + \sum_{i\ge 1}(-1)^i[R^i\trho'_*G] +[L]\ \  \in K_0(E_1|_{X(\sigma)}).\eeq

From \eqref{n25}, we have a commutative diagram of exact sequences
$$\xymatrix{
0\ar[r] & 0\ar[d]\ar[r] & F\ar[d]^{\eta_{F,\trho^*F}}\ar[r]^{\mathrm{id}} & F\ar[r]\ar[d]^\nu & 0\\
0\ar[r] & \trho'_*G\ar[r] & \trho_*\trho^*F\ar[r] & H\ar[r] & 0
}$$
which gives an exact sequence
$$0\to \ker(\eta_{F,\trho^*F})\to \ker(\nu)\to \trho'_*G\to \coker(\eta_{F,\trho^*F})\to \coker(\nu)\to 0$$
of coherent sheaves on $E_1|_{X(\sigma)}$. Hence we have
\beq\label{n29}
[\ker(\eta_{F,\trho^*F})]-[\coker(\eta_{F,\trho^*F})]=[\ker \nu]-[\coker \nu] -[\trho'_*G]\ \ \in K_0(E_1|_{X(\sigma)}).
\eeq
From \eqref{n26}, we have a commutative diagram of exact sequences
$$\xymatrix{
0\ar[r] & F\ar[d]^\nu\ar[r]^{\mathrm{id}}  & F\ar[d]^{\eta_{F,\tF}}\ar[r]& 0\ar[r]\ar[d] & 0\\
0\ar[r] & H\ar[r] &\trho_*\tF\ar[r] & L\ar[r] & 0
}$$
which gives an exact sequence
$$0\to \ker \nu\to  \ker(\eta_{F,\tF})\to 0\to \coker \nu\to \coker(\eta_{F,\tF})\to L\to 0$$
of coherent sheaves on $E_1|_{X(\sigma)}$. Hence we have
\beq\label{n30}
[\ker \nu]-[\coker \nu]=[\ker(\eta_{F,\tF})]-[\coker(\eta_{F,\tF})]+[L]\ \ \in K_0(E_1|_{X(\sigma)}).
\eeq
Upon adding \eqref{n29} with \eqref{n30} and subtracting \eqref{n28}, we obtain
\beq\label{n32} R_{F,\trho^*F}=R_{F,\tF}-\trho'_*[G]\ \ \in K_0(E_1|_{X(\sigma)})\eeq
as desired. This proves (i).

Next we prove (ii). Let $0\to F'\to F\to F''\to 0$ be an exact sequence of coherent sheaves on $E_1(\sigma)$. Since $\trho^*$ is right exact, we have an exact sequence
$$\trho^*F'\lra \trho^*F\lra \trho^*F''\lra 0.$$
Let $\tF'$ be the image of the first arrow and let
$$\tF=\trho^*F,\quad \tF''=\trho^*F''$$
so that we have an exact sequence
\beq\label{n31} 0\lra \tF'\lra \tF\lra \tF''\lra 0\eeq
of coherent sheaves on $\tE'_1$. When applied to the long exact sequence
$$0\to \trho_*\tF'\to \trho_*\tF\to \trho_*\tF''\to R^1\trho_*\tF'\to R^1\trho_*\tF\to R^1\trho_*\tF''\to \cdots$$
associated to \eqref{n31}, the same argument that we used above to prove \eqref{n32} gives us the equality
\beq\label{n33} R_{F',\tF'}+R_{F'',\tF''}=R_{F,\tF}\ \ \in K_0(E_1|_{X(\sigma)}).\eeq
On the other hand, \eqref{n31} gives us the equality $[\tF]=[\tF']+[\tF'']\in K_0(\tE'_1)$ and hence
\beq\label{n34} \rho'_*(D^\vee\cdot 0^!_{\tE'_1}\tF)=\rho'_*(D^\vee\cdot 0^!_{\tE'_1}\tF')+\rho'_*(D^\vee\cdot 0^!_{\tE'_1}\tF'')\ \ \in K_0(X(\sigma)).\eeq
Combining \eqref{n33} and \eqref{n34}, we obtain \eqref{n17}.

Finally we prove \eqref{n12}. By its definition \eqref{n18},
$$\hat{\jmath}_*R_{F,\tF}=[F]-\trho_*[\tF]\ \ \in K_0(E_1)$$
where $\hat{\jmath}=\jmath\circ \him : E_1|_{X(\sigma)}\to E_1$ is the inclusion map. Hence
\beq\label{n35} \imath_*0^!_{E_1|_{X(\sigma)}}R_{F,\tF}=0^!_{E_1}\hat{\jmath}_*R_{F,\tF}=0^!_{E_1}[F]-0^!_{E_1}\trho_*[\tF]\ \ \in K_0(X).\eeq
On the other hand, we have
\beq\label{n36} 
\imath_*\rho'_*(D^\vee\cdot 0^!_{\tE'_1}\tF)=\rho_*\tim_*(D^\vee\cdot 0^!_{\tE'_1}\tF)%=\rho_*(0^!_{\sO_{\tX}(-D)}0^!_{\tE'_1}\tF)
=\rho_*0^!_{\tE_1}[\tF]=0^!_{E_1}\trho_*[\tF]\eeq
where $\tim:D\to \tX$ is the inclusion. Now \eqref{n12} follows from \eqref{n35}, \eqref{n36} and \eqref{n16}. This completes the proof.
\end{proof}

The following is a basic example. 
\begin{exam}\label{n39} Let $X$ be a smooth variety of dimension $n$ and $E$ be a vector bundle on $X$ of rank $n$. Let $\sigma:E\to\sO_X$ be a cosection whose zero locus $X(\sigma)$ consists of one simple point $p$. Then 
\beq\label{n38} 0^!_{E,\sigma}[\sO_X]=(-1)^n[\sO_p] \ \ \in K_0(\{p\}).\eeq
To see it, let $\rho:\tX\to X$ be the blowup of $X$ at $p$. Let $D\cong \PP^{n-1}$ be the exceptional divisor so that we have an exact sequence
$$0\lra \tE'\lra \rho^*E\lra \sO_{\tX}(-D)\lra 0$$
whose restriction to the exceptional divisor is 
$$0\lra \tE'|_{\PP^{n-1}}\lra \sO_{\PP^{n-1}}^{\oplus n}\lra \sO_{\PP^{n-1}}(1)\lra 0$$
so that $\tE'|_{\PP^{n-1}}=T^\vee_{\PP^{n-1}}(1)$. By the Whitney sum formula, $c_{n-1}(\tE'|_{\PP^{n-1}})=(-1)^{n-1}c_1(\sO_{\PP^{n-1}}(1))^{n-1}$ and hence
$$D^\vee\cdot 0^!_{\tE'}[\sO_{\tX}] %=-\PP^{n-1}\cdot  0^!_{\tE'}[\sO_{\tX}]
=(-1)^n[\sO_{pt}].$$
Since $\rho^*\sO_X=\sO_{\tX}$, $\rho_*\rho^*\sO_X=\rho_*\sO_{\tX}=\sO_X$ and $R^i\rho_*\sO_{\tX}=0$ for $i\ge 1$, we have $R_{\sO_X,\sO_{\tX}}=0$ and 
$$0^!_{E,\sigma}[\sO_X]=\rho'_*(D^\vee\cdot 0^!_{\tE'}[\sO_{\tX}])=(-1)^n[\sO_p]\ \ \in K_0(\{p\})$$
by \eqref{n16}.
\end{exam}

\section{Cosection localized virtual structure sheaf}\label{Sclss}

Let $X$ be a \DM stack. Let $\phi:E\to \bbL_X$ be a perfect obstruction theory and 
\[ \sigma:Ob_X=h^1(E^\vee)\lra \sO_X \]
be a homomorphism, called a cosection of the obstruction sheaf $Ob_X$. 
Let $X(\sigma)$ be the zero locus of $\sigma$, i.e. the closed substack of $X$ defined by the ideal sheaf $\sigma(Ob_X)\subset \sO_X$. 

We assume that $E$ admits a global resolution $$[E^{-1}\to E^0]$$ by a 2-term complex of locally free sheaves on $X$. The intrinsic normal cone $\bfc_X$ is canonically embedded into the intrinsic normal sheaf $\fN_X$ which in turn embeds into 
$$\fE_X=h^1/h^0(E^\vee)=[E_1/E_0]$$
by $h^1/h^0(\phi^\vee)$ where $E_i$ is the dual of $E^{-i}$. 
The fiber product
$$\xymatrix{
C_1\ar[r]\ar[d] & E_1\ar[d]\\
\bfc_X \ar[r] & \fE_X }$$
defines a cone $C_1$ in $E_1$ and we have
$$[X]\virt=0^!_{\fE_X}[\bfc_X]=0^!_{E_1}[C_1]\ \ \in A_*(X),$$
$$[\sO\virt_{X}]=0^!_{E_1}[\sO_{C_1}]=[\sO_X\otimes^L_{\sO_{E_1}}\sO_{C_1}]\ \ \in K_0(X).$$

In this section, we prove the following.

\begin{theo}\label{n8}
Suppose a \DM stack $X$ is equipped with a perfect obstruction theory $\phi:E\to \bbL_X$ and a cosection $\sigma:Ob_X\to \sO_X$ whose zero locus is denoted by $X(\sigma)$. Assume $E$ admits a global resolution $[E^{-1}\to E^0]$ by locally free sheaves whose dual is denoted by $[E_0\to E_1]$ and let $C_1=\fC_X\times_{[E_1/E_0]}E_1$ be the lift of the intrinsic normal cone $\fC_X$ to $E_1$, whose support is contained in $E_1(\sigma)$ by \cite[\S4]{KLc}. We define the cosection localized virtual structure sheaf to be
\beq\label{3.43}  [\sO\virt_{X,\loc}]=0^!_{E_1,\underline{\sigma}}[\sO_{C_1}]\in K_0(X(\sigma)) \eeq
where $\underline{\sigma}:E_1\to h^1(E^\vee)=Ob_X\to \sO_X$ is the composition of the quotient map $E_1\to h^1(E^\vee)$ with $\sigma$.  
%(See notation in \eqref{n53}.) 
It satisfies
\beq\label{n9} \imath_*  [\sO\virt_{X,\loc}] = [\sO\virt_X]\in K_0(X) \eeq
where $\imath:X(\sigma)\to X$ denotes the inclusion. 
\end{theo}
We will further show that the cosection localized virtual structure sheaf is deformation invariant and independent of all the choices such as a resolution $[E^{-1}\to E^0]$ of the perfect obstruction theory on $X$.  
\begin{proof}  %We use the notation of Theorem \ref{n10}. 
%By \cite[\S4]{KLc}, the cone $C_1\subset E_1$ has (reduced) support in $E_1(\sigma)$ defined in \eqref{n37}. 
%By Proposition \ref{n1}, $\sO_{C_1}$ gives us a class
%$$[\sO_{C_1}]\in K_0(E_1(\sigma)).$$
%Applying the cosection localized Gysin map constructed in Theorem \ref{n10}, we define the cosection localized virtual structure sheaf by
%\beq\label{n53}
%[\sO\virt_{X,\loc}]=0^!_{E_1,\underline{\sigma}}[\sO_{C_1}]\ \ \in K_0(X(\sigma))\eeq
%where $\underline{\sigma}$ is the composition $E_1\to h^1(E^\vee)=Ob_X\mapright{\sigma} \sO_X$.
%We prove \eqref{n9}. 
Since $[\sO\virt_X]=0^!_{E_1}[\sO_{C_1}]$, by \eqref{n12}, we have 
$$\imath_*[\sO\virt_{X,\loc}]=\imath_*0^!_{E_1,\underline{\sigma}}[\sO_{C_1}]=0^!_{E_1}[\sO_{C_1}]=[\sO\virt_{X}]\in K_0(X)$$
as desired.
\end{proof}

\begin{prop}\label{3.65}
The cosection localized virtual structure sheaf $[\sO\virt_{X,\loc}]\in K_0(X(\sigma))$ is independent of the choice of a global resolution $[E^{-1}\to E^0]$ of the perfect obstruction theory $E$ of $X$. 
\end{prop}
\begin{proof}
By the proof of \cite[Proposition 5.3]{BeFa}, it suffices to consider the case where the two resolutions are related by surjective homomorphisms
$$\xymatrix{
\tE^\vee &[\tE_0\ar[r]\ar@{->>}[d] & \tE_1\ar@{->>}[d]]\\
E^\vee  & [E_0\ar[r] & E_1].
}$$
If we let $C_1\subset E_1$ be the pullback of the intrinsic normal cone $\fC_X\subset E_1/E_0$ of $X$ to $E_1$, then $\tC_1=C_1\times_{E_1}\tE_1$ is the pullback of $\fC_X$ to $\tE_1$. 
%Moreover, if $\sO_{C_1}=\sum_i n_i[\cF_i]$ by Proposition \ref{2.31},  then 
Then $[\sO_{\tC_1}]=q^*[\sO_{C_1}]$ where $q:\tE_1\to E_1$ denotes the projection. 
By the definition of $0^!_{E_1,\sigma}$ above, it is straightforward to see that
$$0^!_{E_1,\underline{\sigma}}[\sO_{C_1}]=0^!_{\tE_1,\underline{\tsi}}[\sO_{\tC_1}]$$
where $\underline{\tsi}$ is the composition $\tE_1\to E_1\mapright{\underline{\sigma}}\sO_X.$
%which implies the desired equality $$ 0^!_{E_1,\mathbf{\sigma}}[\sO_{C_1}] = 0^!_{\tE_1,\mathbf{\sigma}}[\sO_{\tC_1}]. $$
\end{proof}

\begin{rema}
Exactly the same holds in the algebraic cobordism group of $X$ instead of the algebraic K-group $K_0(X)$. See \cite{Shen} for a virtual fundamental class in the algebraic cobordism group of $X$.  We have the cosection localized virtual class $[X]\virt_{\Omega,\loc}$ in the algebraic cobordism group of $X(\sigma)$. Details will appear in a subsequent paper.  
\end{rema}

\begin{defi}\label{n40}
Let $X$ be a \DM stack equipped with a perfect obstruction theory $\phi:E\to \bbL_X$ and a cosection $\sigma:Ob_X=h^1(E^\vee)\to \sO_X$. Let $[\sO\virt_{X,\loc}]\in K_0(X(\sigma))$ be  the cosection localized virtual structure sheaf of $X$. Suppose the zero locus $X(\sigma)$ of the cosection $\sigma$ is proper. The cosection localized \emph{virtual Euler characteristic} of a class $\beta\in K^0(X)$ is defined as
\beq\label{n41} \chi\virt_\loc (X,\beta)=\chi(X(\sigma),\beta\cdot\sO\virt_{X,\loc})=\sum_i (-1)^i \dim H^i(X(\sigma),\beta\cdot \sO\virt_{X,\loc}).\eeq
\end{defi}

We next prove that the cosection localized virtual structure sheaf is deformation invariant. 
Let $X\to S$ be a morphism of stacks, where $X$ is a \DM stack and $S$ is a smooth Artin stack. 
Let $v:Z\to W$ be a regular embedding of schemes that fits into a Cartesian square
$$\begin{CD}
Y @>u>> X\\
@VVV @VVV\\
Z @>v>> W
\end{CD}
$$
%regular embedding and $Y=X\times_ST\to T$.
Suppose we have relative perfect obstruction theories $\phi:E\to \bbL_{X/S}$ and $\phi':E'\to \bbL_{Y/S}$ that fit into a morphism of distinguished triangles
\beq\label{XY}
\begin{CD}
E|_Y @>>> E'@>>> N_{Z/W}^\vee|_Y[1]@>{+1}>>\\
@VVV@VVV@VV{\cong}V\\
\bbL_{X/S}|_Y@>>> \bbL_{Y/S}@>>>\bbL_{Y/X}@>{+1}>>,
\end{CD}
\eeq
%\[\xymatrix{
%E|_Y\ar[d]\ar[r] & F\ar[r] \ar[d]& N_{Z/W}|_Y[1]\ar[d]^{\cong}\ar[r]^{+1}&\\
%\bbL_{X/S}|_Y\ar[r] & \bbL_{Y/S}\ar[r] &\bbL_{Y/X}\ar[r]^{+1}&
%}\]
where $N_{Z/W}$ is the normal bundle of $Z$ in $W$.

As above, we assume the (relative) perfect obstruction theory $E$ admits a global resolution 
\beq\label{nn1} [E^{-1}\to E^0]\eeq
by locally free sheaves $E^{-1}$ and $E^0$ on $X$. Since $E'$ is the cone of the morphism $N_{Z/W}^\vee|_Y\to E|_Y$, $E'$ has the global resolution
\beq\label{nn2} [E^{-1}|_Y\oplus N_{Z/W}^\vee|_Y\to E^0|_Y].\eeq
From the distinguished triangle $\bbL_S|_X\to \bbL_X\to \bbL_{X/S}\to\cdots$, we find that
\beq\label{nn3} [E^{-1}\to E^0\oplus T_S^\vee|_X]\eeq
is an (absolute) perfect obstruction theory of $X$. Similarly, 
\beq\label{nn4} [E^{-1}|_Y\oplus N_{Z/W}^\vee|_Y\to E^0|_Y\oplus T_S^\vee|_Y]\eeq
is an (absolute) perfect obstruction theory of $Y$. Recall that the (absolute) obstruction sheaf $Ob_X$ of $X$ is the cokernel of the dual of \eqref{nn3} and the obstruction sheaf $Ob_Y$ of $Y$ is that of the dual of \eqref{nn4}. Hence we have an exact sequence
\beq\label{nn5} N_{Z/W}|_Y\lra Ob_{Y}\lra Ob_{X}|_Y\lra 0.\eeq

Let $\sigma:Ob_X\to \sO_X$ be a cosection for $X$ and let
$$\sigma':Ob_Y\twoheadrightarrow Ob_X|_Y\mapright{\sigma|_Y} \sO_Y
$$
be the induced cosection for $Y$.
Let $\sO\virt_{X,\loc}$ (resp. $\sO\virt_{Y,\loc}$) be the cosection localized virtual structure sheaf of $X$ (resp. $Y$) with respect to the perfect obstruction theory \eqref{nn3} of $X$ (resp. \eqref{nn4} of $Y$). %induced by $\phi$ (resp. $\varphi$). 
Let $X(\sigma)$ (resp. $Y(\sigma)$) denote the zero locus of the cosection $\sigma$ and (resp. $\sigma'$). Then by construction, $Y(\sigma)=X(\sigma)\times_XY$ which fits into the Cartesian square 
%We denote the inclusion $t: Y(\sigma')\to X(\sigma)$.
%Because $t$ fits into the Cartesian square
$$\begin{CD}Y(\sigma) @>t>> X(\sigma)\\
@VVV @VVV\\
Z @>v>> W,
\end{CD}
$$
since $Y=Z\times_WX$. 
We then have the Gysin map $$v^!: K_0(X(\sigma))\to K_0(Y(\sigma))$$ defined by
\beq\label{nn6} v^![\sO_A]=0_{N_{Z/W}}^![ \sO_{C_{Y(\sigma)\cap A/A}}]\in K_0(Y(\sigma)),\quad A\sub X(\sigma).
\eeq
In fact, if we choose a finite locally free resolution $\sO_Z^W$ of $v_*\sO_Z$ on $W$, then the Gysin map $v^!$ is given by the tensor product
\beq\label{nn8}
v^![F]=[\sO_Z^W|_{X(\sigma)}\otimes_{\sO_{X(\sigma)}}F],\quad F\in K_0(X(\sigma))
\eeq
by \cite[Remark 1]{YPLe}.

To avoid discussion about K-theory of Artin stacks, let us assume that there is a smooth morphism $M\to S$ and a closed embedding $X\subset M$ over $S$. 
This is always possible if, for instance, $X\to S$ is quasi-projective.

\begin{prop}\label{3.66}
Under the above assumptions, 
we have $$v^![\sO_{X,\loc}\virt]=[\sO_{Y,\loc}\virt]\in K_0(Y(\sigma)).$$
\end{prop}

\begin{proof}
The proof essentially follows from that of \cite[Theorem 5.2]{KLc}, using the work of \cite{KKP}.
We outline the arguments. 
By \cite{Fulton}, there is a deformation $\cM\to \PP^1$ whose fibers are $M$ except the central fiber $C_{X/M}$ 
which is the normal cone of $X$ in $M$. In \cite[below (7)]{KKP}, the authors constructed a double deformation space
over $\Po\times\Po$, that is the deformation of $\cM$ to the normal cone 
$C_{Y\times\PP^1/\cM}$ of $Y\times \PP^1$. We denote this double deformation space by $\cW$.

By its construction, $\cW$ is flat over $\Po\times\Po-\{(0,0)\}$, where $(0,0)\in\Po\times\Po$ is the special point
having the following properties: The fiber of $\cW$ over $(1,0)\in \Po\times\Po$ is $C_{Y/M}$;
the flat specialization along $\Po\times\{0\}$ (over $(0,0)$) is $C_{Y/C_{X/M}}$.

Because $X$ is of finite type, an easy argument shows that we can find a smooth birational model
$U\to\Po\times\Po$, isomorphic over $\Po\times\Po-\{(0,0)\}$, so that 
$$\cW\times_{\Po\times\Po}\bl\Po\times\Po-\{(0,0)\}\br \times_{\Po\times\Po}U
$$
extends to a family $\tilde \cW\sub \cW\times_{\Po\times\Po}U$, flat over $U$.
Consequently, we can find a chain $\Sigma=\cup_{i=1}^n \Sigma_i$ of $\Po$'s in $U$,
lying over $\Po\times \{0\}$, and two points $a\in \Sigma_1$ and $b\in\Sigma_n$ so that
$$\tilde\cW|_a=C_{Y/M},\quad \tilde\cW|_b=C_{Y/C_{X/M}}.
$$
We thus obtain a rational equivalence
\beq\label{nn10}
[{C_{Y/M}}]=[{C_{Y/C_{X/M}}}] .
\eeq

%The relative perfect obstruction theory $E$ on $X$ gives us an absolute perfect obstruction theory on $X$ by adding $T^\vee_S=\bbL_S$ (cf. \eqref{n42}). Let $[E_0\to E_1]$ be a global resolution of the dual of the absolute perfect obstruction theory of $X$. Likewise $F$ induces an absolute perfect obstruction theory on $Y$ by adding $T^*_S$. Let $[F_0\to F_1]$ denote a global resolution of the dual of the absolute perfect obstruction theory of $Y$.  
Recall that \eqref{nn3} and \eqref{nn4} are perfect obstruction theories for $X$ and $Y$ respectively. Let $E_i$ denote the dual of $E^{-i}$ for $i=0,1$. 
Since the intrinsic normal cone is $\bfc_X=C_{X/M}/T_M|_X$, the cone 
$C_X:=C_1=\bfc_X\times_{\cE_X}E_1$ equals 
$$C_{X/M}\times_X(E_0\oplus T_S|_X)/T_M|_X\subset E_1(\underline{\sigma})$$
where $\underline{\sigma}:E_1\twoheadrightarrow Ob_X\mapright{\sigma} \sO_X$ is the cosection of $E_1$ induced by $\sigma$.  
Hence we have
\beq\label{nn11}[\sO_{X,\loc}\virt]=0^!_{E_1,\underline{\sigma}}[\sO_{C_{X/M}\times_X(E_0\oplus T_S|_X)/T_M|_X}]\in K_0(X(\sigma)).\eeq

By \eqref{nn5}, the restriction $\underline{\sigma}':E_1|_Y\to\sO_Y$ of $\underline{\sigma}$ factors through $\sigma':Ob_Y\to \sO_Y$ defined above. 
By definition, we have the equality $E_1(\underline{\sigma})|_Y=E_1|_Y(\underline{\sigma}')$ and a Cartesian diagram
\[\xymatrix{
E_1(\underline{\sigma})|_Y\ar[r]\ar[d] & E_1(\underline{\sigma})\ar[d]\\
Y\ar[d]\ar[r] & X\ar[d]\\
Z\ar[r]^v & W
}\]
Similarly to the case of $X$ above, the cone $C_Y=\bfc_Y\times_{\cE_Y}(E_1|_Y\oplus N_{Z/W}|_Y)$ equals 
$$C_{Y/M}\times_Y (E_0|_Y\oplus T_S|_Y)/T_{M}|_Y\subset E_1(\underline{\sigma})|_Y\oplus N_{Z/W}|_Y
$$
so that 
\beq\label{nn12} [\sO_{Y,\loc}\virt]=0^!_{E_1|_Y\oplus N_{Z/W}|_Y,\underline{\sigma}'}[\sO_{C_{Y/M}\times_Y (E_0|_Y\oplus T_S|_Y)/T_{M}|_Y}]\eeq
$$
=0^!_{E_1|_Y,\underline{\sigma}'}0^!_{N_{Z/W}|_Y}[\sO_{C_{Y/M}\times_Y (E_0|_Y\oplus T_S|_Y)/T_{M}|_Y}]
$$
since $N_{Z/W}|_Y$ lies in the kernel of $\sigma'$ by \eqref{nn5}.

It was proved in \cite[\S5]{KLc} that the rational equivalence \eqref{nn10} takes place in the kernel of a cosection extended over the double deformation space. Therefore \eqref{nn12} equals 
$$0^!_{E_1|_Y,\underline{\sigma}'}0^!_{N_{Z/W}|_Y}[\sO_{C_{Y/C_{X/M}}\times_Y (E_0|_Y\oplus T_S|_Y)/T_{M}|_Y}]$$
$$=0^!_{E_1|_Y,\underline{\sigma}'}v^![\sO_{C_{X/M}\times_X (E_0\oplus T_S|_X)/T_{M}|_X}].$$
By Lemma \ref{nn14} below, we have $0^!_{E_1|_Y,\underline{\sigma}'}\circ v^! =v^!\circ 0^!_{E_1,\underline{\sigma}}.$ Hence, $ [\sO_{Y,\loc}\virt]$ equals
$$v^!0^!_{E_1,\underline{\sigma}}[\sO_{C_{X/M}\times_X (E_0\oplus T_S|_X)/T_{M}|_X}]=v^![\sO_{X,\loc}\virt]$$
by \eqref{nn11}. 
This proves the proposition.
\end{proof}

It remains to prove the following.
\begin{lemm}\label{nn14}
With the notation as above, $0^!_{E_1|_Y,\underline{\sigma}'}\circ v^! =v^!\circ 0^!_{E_1,\underline{\sigma}}.$
\end{lemm}
\begin{proof}
Let $\rho:\tX\to X$ denote the blowup of $X$ along $X(\sigma)$ with exceptional divisor $D$ so that we have an exact sequence 
$$0\lra \tE_1'\lra \tE_1\lra \sO_{\tX}(-D)\lra 0$$
where $\tE_1=\rho^*E_1$. Let $A\subset E_1(\sigma)$ be an integral substack. 
If $A\subset E_1|_{X(\sigma)}$, the lemma simply says $v^!\circ 0^!_{E_1}=0^!_{E_1|_Y}\circ v^!$ which is straightforward since both $v^!$ and $0^!_{E_1}$ are tensor products (cf. \eqref{nn8}). 
If $A\not\subset E_1|_{X(\sigma)}$, let $\tilde{A}\subset \tE_1'$ be the proper transform of $A$ so that
$$0^!_{E_1,\underline{\sigma}}[\sO_A]=\rho'_*(D^\vee\cdot 0^!_{\tE_1'}[\sO_{\tilde{A}}])+0^!_{E_1|_{X(\sigma)}}[\sO_A\to \rho_*\sO_{\tilde{A}}].$$
Since $v^![\sO_A]=[\sO^W_Z|_{E_1(\underline{\sigma})}\otimes \sO_A]$ where $\sO_Z^W$ is a finite locally free resolution of $v_*\sO_Z$ on $W$, we have
$$0^!_{E_1|_Y,\underline{\sigma}'} v^![\sO_A]=\rho'_*(D^\vee\cdot 0^!_{\tE_1'}[\sO^W_Z|_{\tE_1'}\otimes \sO_{\tilde{A}}])+0^!_{E_1|_{X(\sigma)}}[\sO^W_Z|_{E_1|_{X(\sigma)}}\otimes [\sO_A\to \rho_*\sO_{\tilde{A}}]].$$
Since $\sO_Z^W$ is a complex of locally free sheaves on $W$ and the Gysin maps are tensor products, we can pull out  $\sO_Z^W$ to obtain 
$$0^!_{E_1|_Y,\underline{\sigma}'} v^![\sO_A]=\sO^W_Z|_{X(\sigma)}\otimes\rho'_*(D^\vee\cdot 0^!_{\tE_1'}[\sO_{\tilde{A}}])+\sO^W_Z|_{X(\sigma)}\otimes 0^!_{E_1|_{X(\sigma)}}[\sO_A\to \rho_*\sO_{\tilde{A}}]$$
$$=\sO^W_Z|_{X(\sigma)}\otimes 0^!_{E_1,\underline{\sigma}} [\sO_A] =
v^!0^!_{E_1,\underline{\sigma}} [\sO_A]$$
as desired.
\end{proof}

As an application of Proposition \ref{3.66}, we will prove a principle of conservation of numbers for the cosection localized virtual Euler characteristics. Let $t\in \PP^1$ be a closed point and $\pi:X\to \PP^1$ be a morphism of \DM stacks, equipped with a relative perfect obstruction theory 
$$\phi:E\to \bbL_{X/\PP^1}.$$
Let $\sigma:Ob_{X/\PP^1}=h^1(E^\vee)\to \sO_X$ be a cosection of the relative obstruction sheaf whose zero locus is denoted by $X(\sigma)$. We assume that $X(\sigma)$ is proper over $\PP^1$
although $X$ may not be proper over $\PP^1$. 
%We also assume that $E$ admits a global resolution $[E^{-1}\to E^0]$ by locally free sheaves on $X$ so that we can apply. 
The restriction 
$$\phi_t:E_t\lra \bbL_{X_t}$$
of $\phi$ to $X_t=t\times_{\PP^1} X$ with $E_t=E|_{X_t}$ is a perfect obstruction theory of $X_t$ and we have an absolute perfect obstruction theory
$$\bar\phi:\bar E\lra \bbL_X$$
of $X$ defined by the commutative diagram
\beq\label{n42}
\xymatrix{
\sO_X\ar[r]\ar@{=}[d] &\bar E\ar[d]^{\bar\phi} \ar[r] & E\ar[r]\ar[d]^\phi & \\
\pi^*\bbL_{\PP^1}\ar[r] & \bbL_X\ar[r] & \bbL_{X/\PP^1}\ar[r] &
}\eeq
of distinguished triangles. The first row of \eqref{n42} gives us an exact sequence
\beq\label{n43} \sO_X\lra Ob_{X/\PP^1}=h^1(E^\vee)\lra Ob_X=h^1(\bar E^\vee)\lra 0.\eeq
We further assume that the cosection $\sigma:Ob_{X/\PP^1}\to \sO_X$ descends to a cosection
$$\bar\sigma:Ob_X\to \sO_X$$
so that we have the cosection localized virtual structure sheaf $[\sO\virt_{X,\loc}]\in K_0(X(\bar\sigma)).$ 
Likewise, the homomorphism $\sigma_t:Ob_{X_t}\to \sO_{X_t}$ induced by $\sigma$ gives us the cosection localized virtual structure sheaf $[\sO\virt_{X_t,\loc}]\in K_0(X_t(\sigma_t)).$ 

\begin{coro}\label{n45}
Under the above assumptions, for any $\beta\in K^0(X(\bar\sigma))$
\beq\label{z2}\chi\virt_\loc(X_t,\imath_t^*\beta)=\chi(X_t(\sigma_t),\sO_{X_t,\loc}\virt\cdot\imath_t^*\beta)=\chi(X(\bar\sigma),{\imath_t}_*\sO_{X_t,\loc}\virt\cdot \beta)\eeq
is independent of $t\in \PP^1$, where $\imath_t:X_t(\sigma_t)\hookrightarrow X(\bar\sigma)$.
% denotes the inclusion map.
\end{coro}
\begin{proof}
By Proposition \ref{3.66}, we have
$$t^![\sO_{X,\loc}\virt]=[\sO_{X_t,\loc}\virt] \ \ \in K_0(X_t(\sigma_t))$$
where $t:\{t\}\to \PP^1$ is the inclusion. Then we have 
$${\imath_t}_*[\sO_{X_t,\loc}\virt] = {\imath_t}_*t^![\sO_{X,\loc}\virt]\ \ \in K_0(X(\bar\sigma)).$$
Since $[\sO_t]\in K_0(\PP^1)$ is independent of $t\in \PP^1$, 
${\imath_t}_*t^![\sO_{X,\loc}\virt]$ is independent of $t$. Hence $$\chi(X(\bar\sigma),{\imath_t}_*\sO_{X_t,\loc}\virt\cdot \beta),\quad \beta\in K^0(X(\bar\sigma))$$ is independent of $t$ as desired.
\end{proof}

%\begin{rema}
Another way to prove Corollary \ref{n45} is to use a cosection localized Riemann-Roch as outlined below. For schemes, we have the following cosection localized version of virtual Grothendieck-Riemann-Roch (cf. \cite[Theorem 3.3]{FaGo}).

\begin{theo}\label{z3} Let $f:X\to Y$ be a morphism of schemes with $Y$ smooth. Let $\phi:E\to\bbL_X$ be a perfect obstruction theory of $X$ and $\sigma:Ob_X\to \sO_X$ be a cosection. Suppose the restriction $f':X(\sigma)\to Y$ of $f$ to the zero locus $X(\sigma)$ of the cosection $\sigma$ is proper. Let $T_X\virt=[E_0]-[E_1]\in K^0(X)$ be the virtual tangent bundle where $[E_0\to E_1]$ is the dual of $E$. 
Then for $\beta\in K^0(X)$, we have
\beq\label{3.47}\ch(f'_*(\beta\cdot\sO\virt_{X,\loc})) \td(T_Y)\cap [Y]=f'_*(\ch(\beta) \td(T_X\virt)\cap [X]\virt_\loc).\eeq
In particular, if $Y$ is a point and $X(\sigma)$ is proper, the cosection localized virtual Euler characteristic of $\beta$ is 
\beq\label{3.48}\chi\virt_\loc(X,\beta)=\int_{[X]\virt_\loc}\ch(\beta)\td(T_X\virt).\eeq
\end{theo}

%We omit the proof since we are not going to use it. 

For a quasi-projective \DM stack $X$, we choose a 
smooth projective \DM stack $M$ and a closed immersion $X\hookrightarrow M$. Then \eqref{z2} is 
\beq\label{n52}\chi\virt_\loc(X_t,\imath_t^*\beta)=\chi(M,\jmath_*{\imath_t}_*\sO\virt_{X_t,\loc}\cdot \beta), \quad \beta\in K^0(M)\eeq
where $\jmath:X(\bar{\sigma})\hookrightarrow X\hookrightarrow M$ is the inclusion. 
Applying the Kawasaki-Riemann-Roch for the smooth \DM stack $M$ together with \eqref{3.47}, we can express the right side of \eqref{n52} in terms of cosection localized virtual integrals on inertia substacks $X_\mu$. Since the cosection localized virtual fundamental classes are deformation invariant by \cite{KLc}, we find that the left side of \eqref{n52} is independent of $t$. 
%\end{rema}

We complete this section with a proof of \eqref{3.47}.
\begin{proof}[Proof of Theorem \ref{z3}]
The theorem is proved by copying the proof of \cite[Theorem 3.3]{FaGo} line by line, if we replace the ordinary Gysin map $0^!_{E_1}$ by the cosection localized Gysin map $0^!_{E_1,\sigma}$. Perhaps the only nontrivial fact to be checked is the identity
\beq\label{z4} \tau_{X(\sigma)}(0^!_{E_1,\sigma}[F])=\td(-E_1)\cap 0^!_{E_1,\sigma}(\tau_{E_1(\sigma)}[F])\eeq
for a coherent sheaf $F$ on $E_1(\sigma)$. 

Recall that the cosection localized Gysin map for Chow groups 
$$0^!_{E_1,\sigma}:A_*(E_1(\sigma))\to A_*(X(\sigma))$$ is defined in \cite{KLc} by 
\beq\label{z5} 0^!_{E_1,\sigma}(\xi)=\rho'_*(-D\cdot 0^!_{\tE_1'}(\zeta))+0^!_{E_1|_{X(\sigma)}}(\eta)\eeq
whenever $\xi=\trho_*\zeta+\hat{\imath}_*\eta$ for $\zeta\in A_*(\tE'_1)$ and $\eta\in A_*(E_1|_{X(\sigma)})$.
Here $\hat{\imath}:E|_{X(\sigma)}\to E_1(\sigma)$ is the inclusion and $\rho:\tX\to X$ is the blowup of $X$ along $X(\sigma)$ while $\rho':D\to X(\sigma)$ is the restriction of $\rho$ to the exceptional divisor $D$ as in \S\ref{SclG}. 

For any choice of $\tF$ as in \eqref{n15}, by \eqref{n18}, we have $F=R\trho_*\tF+R\hat{\imath}_*R_{F,\tF}$ with $[\tF]\in K_0(\tE'_1)$ and $[R_{F,\tF}]\in K_0(E_1|_{X(\sigma)}).$ Hence, we have 
$$\tau_{E_1(\sigma)}[F]=\tau_{E_1(\sigma)}R\trho_*\tF+\tau_{E_1(\sigma)}R\hat{\imath}_*R_{F,\tF}$$
\beq\label{z6} = \trho_*\tau_{\tE'_1}[\tF]+\hat{\imath}_*\tau_{E_1|_{X(\sigma)}}R_{F,\tF}.\eeq
By \eqref{z5}, we thus have
\beq\label{z7}
0^!_{E_1,\sigma}(\tau_{E_1(\sigma)}[F])=\rho'_*(-D\cdot 0^!_{\tE_1'}(\tau_{\tE'_1}[\tF]))+0^!_{E_1|_{X(\sigma)}}(\tau_{E_1|_{X(\sigma)}}R_{F,\tF}).
\eeq

By \eqref{n16} and \eqref{z7}, we have 
$$\tau_{X(\sigma)} 0^!_{E_1,\sigma}[F]=\tau_{X(\sigma)}\rho'_*(D^\vee\cdot 0^!_{\tE'_1}\tF)+\tau_{X(\sigma)}0^!_{E_1|_{X(\sigma)}}R_{F,\tF}$$
$$=\rho'_*\tau_D(\sO_D^\vee\otimes 0^!_{\tE'_1}\tF)+\td(-E_1) 0^!_{E_1|_{X(\sigma)}}(\tau_{E_1|_{X(\sigma)}}R_{F,\tF})$$
$$=\rho'_* \ch(\sO_D^\vee)\tau_D(0^!_{\tE'_1}\tF)+\td(-E_1) 0^!_{E_1|_{X(\sigma)}}(\tau_{E_1|_{X(\sigma)}}R_{F,\tF})$$
$$=\rho'_* \ch(\sO_D^\vee)\td(-\tE'_1) 0^!_{\tE'_1}(\tau_{\tE'_1}\tF)+\td(-E_1) 0^!_{E_1|_{X(\sigma)}}(\tau_{E_1|_{X(\sigma)}}R_{F,\tF})$$
$$=\td(-E_1) \left(\rho'_* (\ch(\sO_D^\vee)\td(\sO_{\tX}(-D)) 0^!_{\tE'_1}(\tau_{\tE'_1}\tF))+0^!_{E_1|_{X(\sigma)}}(\tau_{E_1|_{X(\sigma)}}R_{F,\tF}) \right)$$
$$=\td(-E_1) \left(\rho'_* (-D\cdot 0^!_{\tE'_1}(\tau_{\tE'_1}\tF))+0^!_{E_1|_{X(\sigma)}}(\tau_{E_1|_{X(\sigma)}}R_{F,\tF}) \right)$$
$$=\td(-E_1) 0^!_{E_1,\sigma}(\tau_{E_1(\sigma)}[F])$$
because $\ch(\sO_D^\vee)\td(\sO_{\tX}(-D))=-D$ as $\ch(\sO_D^\vee)=\ch[\sO_{\tX}\to \sO_{\tX}(D)]=1-e^D$. This proves \eqref{z4} and the theorem.
\end{proof}

%\begin{prop}
%\begin{enumerate}
%\item The cosection localized virtual structure sheaf $\sO\virt_{X,\loc}$ defined by Theorem \ref{3.10} is indepedent of the choice of a resolution $[E^{-1}\to E^0]$ of $E$ by locally free sheaves.
%\item Let $f:X'\to X$ be a morphism of schemes $X$ and $X'$ with perfect obstruction theories $E\to \bbL_X$ and $E'\to \bbL_{X'}$ which fit into a commutative diagram
%$$\xymatrix{
%f^*E\ar[r]\ar[d] &E'\ar[d]\ar[r] & E''\ar[r[]\ar[d] & f^*E[1]\ar[d]\\
%f^*\bbL_{X}\ar[r] & \bbL_{X'} \ar[r] &\bbL_{X'/X}\ar[r] f^*\bbL_X[1]
%}$$
%of distinguished triangles. Suppose $E''$ is also perfect of amplitute $[-1,0]$. Then $f^*
%\end{enumerate}
%\end{prop}

%\begin{rema} Proposition \ref{3.2} and Theorem \ref{3.10} hold for an \DM stack $X$, which is virtually smooth over a smooth Artin stack $S$. \end{rema}

\section{Application to K-theoretic Landau-Ginzburg/Calabi-Yau correspondence}\label{SLG}

In this section, we apply the cosection localized virtual structure sheaf to a K-theoretic Landau-Ginzburg/Calabi-Yau correspondence. 

\subsection{K-theoretic FJRW invariant}\label{SLG1}
The cosection localized virtual structure sheaf enables us to define the K-theoretic FJRW invariant for \emph{narrow} sector. To simplify the discussion, we focus on the Fermat quintic case
\beq\label{n55} \sum_{i=1}^5z_i^5 : \CC^5\lra \CC.\eeq

Let $S$ be the moduli space of $5$-spin curves, i.e. triples $(C,L,p)$ of \begin{enumerate}
\item a 1-dimensional projective \DM stack $C$ with at worst nodal singularities whose stabilizer groups at nodes or marked points are $\ZZ_5$ which acts on a node $(zw=0)$ as $\zeta\cdot(z,w)=(\zeta z,\zeta^{-1} w)$ for $\zeta^5=1$,
\item a line bundle $L$ on $C$, and
\item an isomorphism $p:L^5\to \omega_C^{\mathrm{log}}.$
\end{enumerate}
Assume all the marked points are narrow, i.e. the action of the stabilizer group $\ZZ_5$ on the fiber of $L$ at a marked point is not trivial. 

Let $X$ be the \DM stack of quadruples $(C,L,p,x)$ where $(C,L,p)\in S$ and $x\in H^0(L)^{\oplus 5}$. By \cite{CLp, CLL}, there is a perfect obstruction theory on $X$ and a cosection $\sigma:Ob_X\to \sO_X$ whose zero locus $X(\sigma)$ is the proper \DM stack $S$. 
As the relative obstruction theory of $X/S$ can be presented by a two-term complex of vector bundles on
$X$, we can apply Theorem \ref{n8} to obtain the cosection localized virtual structure sheaf
\beq\label{n48} [\sO\virt_{X,\loc}]\in K_0(S).\eeq
%We may think of $ [\sO\virt_{X,\loc}]$ as the K-theoretic FJRW invariant in algebraic geometry. 

To obtain numerical invariants, we take the Euler characteristic.
\begin{defi}\label{n47} The K-theoretic Fan-Jarvis-Ruan-Witten invariant is defined by
$$\chi\virt_{\loc}(X,\beta)=\chi(S,\sO\virt_{X,\loc}\cdot\beta),\quad \beta\in K^0(S).$$
\end{defi}

Since $S$ is a smooth \DM stack, the Kawasaki-Riemann-Roch theorem enables us to express the numerical K-theoretic FJRW invariant as (cosection localized) virtual integrals. 

\medskip

In \cite{Chio}, based on the Polishchuk-Vaintrob construction of Witten's top Chern class \cite{PoVa}, Chiodo constructed a K-theory class $Ke(E^\vee,\tau^\vee)\in K_0(S)$ whose top Chern class $$c_{\mathrm{top}}(E,\tau)=\mathrm{ch}(Ke(E^\vee,\tau^\vee))\mathrm{td}(E)^{-1}\in A_*(S)$$
coincides with the cosection localized virtual fundamental class $[X]\virt_{\loc}$ in $A_*(S)$ by \cite[\S5.2]{CLL} and \cite[Theorem 5.4.1]{Chio}. We have the following comparison result. 
\begin{prop} Chiodo's K-theory class $Ke(E^\vee,\tau^\vee)$ coincides with the cosection localized virtual structure sheaf $[\sO\virt_{X,\loc}]$ in $K_0(S)$. 
\end{prop}
\begin{proof}
The same argument as in the proof of \cite[Proposition 5.10]{CLL} also proves this proposition. Indeed, since the pushforward and the localized Gysin map $0^!_{E_1,\sigma}$ commute, 
% $0^!_{E_1,\sigma}R\rho_*\sO_{\tilde{C}}=\rho'_*(D\cdot 0^!_{\tE_1'}\sO_{\tilde{C}})=\rho'_*0^!_{\tE_1,\tilde{\sigma}}[\sO_{\tilde{C}}]$ since $\eta_{R\rho_*\sO_{\tilde{C}},\sO_{\tilde{C}}}=0.$
it suffices to prove the proposition on the blowup $\tX$ of $X$ along $X(\sigma)$ where we have an exact sequence
$$0\lra \tE_1'\lra \tE_1\lra \sO_{\tX}(-D)\lra 0.$$
Here $D$ denotes the exceptional divisor. 
By deforming this exact sequence to the split case $\tE_1=\tE_1'\oplus \sO_{\tX}(-D)$, the proposition is reduced to the case of vanishing cosection $\tE_1'\to 0$ and the case where cosection is the natural inclusion $\sO_{\tX}(-D)\to \sO_{\tX}$. Each of these cases is easy to check. 
\end{proof}

\medskip

\subsection{GSW model for K-theoretic Gromov-Witten invariant}\label{SLG2}

Consider the Fermat quintic Calabi-Yau 3-fold
 $$Y=(\sum_{i=1}^5z_i^5=0) \subset \PP^4.$$
Let $Z$ be the deformation of $\PP^4$ to the normal cone $\sO_Y(5)$ of $Y$ in $\PP^4$, i.e. $Z$ is the complement of the proper transform of $\{0\}\times \PP^4$ in the blowup of $\PP^1\times \PP^4$ along $\{0\}\times Y$. Let
$$p:Z\lra \PP^1\times \PP^4\lra \PP^1$$
denote the composition. Then $Z_t=t\times_{\PP^1}Z$ is $\PP^4$ for $t\ne 0$ and $\sO_Y(5)$ for $t=0$. 

Let $N=\overline{M}_g(Y,d)$ (resp. $M=\overline{M}_g(\PP^4,d)$) be the moduli space of stable maps to $Y$ (resp. $\PP^4$). Let $M^p=\overline{M}_g(\PP^4,d)^p$ denote the moduli space of pairs $(f,p)$ where $(f:C\to \PP^4)\in M$ and $p\in H^0(f^*\sO_{\PP^4}(-5)\otimes\omega_C)$. 

Let $\bar{X}=\overline{M}_g(Z/\PP^1,d)$ be the moduli space of stable maps of genus $g$ and degree $d$ to the fibers of $p$. Let $$\pi:X=\overline{M}_g(Z/\PP^1,d)^p\lra \PP^1$$
denote the moduli space of pairs $(f,p)$ where $(f:C\to Z)\in \bar{X}$ and $p\in H^0(f^*\sO_{\PP^4}(-5)\otimes\omega_C)$ (cf. \cite[\S4.1]{CLp}). 
It is straightforward that $X_t$ is $M^p$ for $t\ne 0$ and the central fiber $X_0$ is the moduli space of triples $(f,s,p)$ where $(f:C\to Y)\in N=\overline{M}_g(Y,d)$, $s\in H^0(f^*\sO_Y(5))$ and $p\in  H^0(f^*\sO_{Y}(-5)\otimes\omega_C)$.
By \cite[\S4]{CLp}, all the assumptions of Proposition \ref{3.66} are satisfied and hence we find that 
$t^![\sO\virt_{X,\loc}]=[\sO\virt_{X_t,\loc}]\in K_0(N)$ for all $t\in \PP^1$ and
$$\chi(N,\sO\virt_{M^p,\loc})=\chi(N, \sO\virt_{X_0,\loc}).$$
Here $X(\sigma)=N\times\PP^1$.  
On the other hand, the proof of \cite[Theorem 5.7]{CLp} together with Example \ref{n39} above proves 
$$\chi(N, \sO\virt_{X_0,\loc})=(-1)^{5d-g+1}\chi(N,\sO\virt_N)=(-1)^{5d-g+1}\chi\virt(N).$$
We therefore proved the following.
\begin{prop}\label{n56} The cosection localized virtual Euler characteristic
$$\chi\virt_{\loc}(\overline{M}_g(\PP^4,d)^p):=\chi(N,\sO\virt_{M^p,\loc})$$ of $M^p$ is equal to the K-theoretic Gromov-Witten invariant $\chi(N,\sO\virt_N)$ of $N$ up to sign by
$$\chi\virt_{\loc}(\overline{M}_g(\PP^4,d)^p)=(-1)^{5d-g+1}\chi\virt(N).$$
\end{prop}

The Landau-Ginzburg/Calabi-Yau correspondence predicts that the FJRW invariant of the Landau-Ginzburg model \eqref{n55} is equivalent to the GW invariant of the Calabi-Yau 3-fold $Y$ via variable changes, analytic continuations, and symplectic transformations. 
Proposition \ref{n56} may be useful for the  K-theoretic LG/CY correspondence.

%%%%%%%%%%%%%%%%%%%%%%%%%%%%%%%%%%%%%%%%%%%%%%%%%%%%%%%%%%%%%%%%%%%%%%%%%%%%%%%%%%%%%%%%%%%%%%

\bibliographystyle{amsplain}

\end{document}